\documentclass[12pt,a4paper]{amsart}
\usepackage{amssymb}
\usepackage{amsfonts}
\usepackage[top=35mm, bottom=35mm, left=30mm, right=30mm]{geometry}
\usepackage[colorlinks=true,citecolor=blue]{hyperref}
\usepackage{mathptmx}
\usepackage{eucal}
\usepackage{graphicx}
\usepackage{mathrsfs}
\usepackage{amssymb}
\usepackage{amsmath}
\usepackage{amsthm}
\usepackage{enumerate}
\usepackage{verbatim}
\usepackage[labelformat=simple]{subcaption}

\usepackage{xcolor}
\newtheorem{thm}{Theorem}[section]
\newtheorem{cor}[thm]{Corollary}
\newtheorem{que}[thm]{Question}
\newtheorem{lem}[thm]{Lemma}
\newtheorem{prop}[thm]{Proposition}

\theoremstyle{definition}

\newtheorem{exam}[thm]{Example}
\newtheorem{rem}[thm]{Remark}
\numberwithin{equation}{section}

\newcommand{\eps}{\varepsilon}
\newcommand{\N}{\mathbb{N}}

\newcommand{\R}{\mathbb{R}}
\newcommand{\C}{\mathbb{C}}

\newcommand{\orb}{\mathrm{Orb}}

\newcommand{\htop}{h_{\mathrm{top}}}

\begin{document}
\title{Quasi-graphs, zero entropy and measures with discrete spectrum}
\author[J. Li]{Jian Li}
\address[J. Li]{Department of Mathematics,
Shantou University, Shantou 515063, Guangdong, China}
\email{lijian09@mail.ustc.edu.cn}

\author[P. Oprocha]{Piotr Oprocha}
\address[P. Oprocha]{AGH University of Science and Technology, Faculty of Applied
    Mathematics, al.
    Mickiewicza 30, 30-059 Krak\'ow, Poland
    -- and --
    National Supercomputing Centre IT4Innovations, Division of the University of Ostrava,
    Institute for Research and Applications of Fuzzy Modeling,
    30. dubna 22, 70103 Ostrava,
    Czech Republic}
\email{oprocha@agh.edu.pl}

\author[G. Zhang]{Guohua Zhang}
\address[G. Zhang]{School of Mathematical Sciences and Shanghai Center for Mathematical Sciences, Fudan University, Shanghai 200433, China}
\email{chiaths.zhang@gmail.com}

\date{\today}

\begin{abstract}
In this paper, we study dynamics of maps on quasi-graphs characterizing their invariant measures. In particular, we prove that  every invariant measure
of quasi-graph map with zero topological entropy has discrete spectrum.
%
%
%In this paper, we study dynamics of maps on quasi-graphs.
%Precisely, we characterize their invariant measures as convex combinations of finitely many invariant measures for some topological graph maps, consequently, we show that  every invariant measure
%of a quasi-graph map with zero topological entropy has discrete spectrum and
%every ergodic invariant measure of a
%quasi-graph map is isomorphic to an ergodic invariant measure of a map on some topological graph.
Additionally, we obtain an analog of Llibre-Misiurewicz's result relating positive topological entropy
with existence of topological horseshoes.

%As a consequence, Sarnak's M\"obius Disjointness Conjecture restricted to the class of quasi-graph maps with zero topological entropy is reduced to already known cases.

We also study dynamics on dendrites and show that
if a continuous map on a dendrite, whose set of all endpoints is closed and has only finitely many accumulation points,
has zero topological entropy, then  every invariant measure supported on an orbit closure has  discrete spectrum.

%We observe that answering the Sarnak's M\"obius Disjointness problem for all maps on dendrites with zero topological entropy is equivalent to solving it for all dynamical systems with zero topological entropy.
%therefore some special classes of one-dimensional continua still can be useful for better understanding of the conjecture.
%%
%%Furthermore, as a byproduct of the above characterization,
%% we explore Sarnak's M\"obius Disjointness Conjecture for one-dimensional maps.
%%We find
%%that solving it
%%restricted to the class of quasi-graph maps with zero topological entropy is reduced to already known cases, however, answering M\"obius Disjointness Conjecture for all maps on dendrites with zero topological entropy is equivalent to solving it for all dynamical systems with zero topological entropy. Thus it gives us a hope that new insight into the conjecture can be obtained by investigating maps in a class of one-dimensional continua.
\end{abstract}

\maketitle

\section{Introduction}

By a \textit{topological dynamical system} or just \textit{dynamical system}, we mean a pair $(X, T)$, where $X$
is a compact metric space with a metric $d$ and $T\colon X\to X$ is a continuous map.

The motivation to consider ``graph-like'' spaces and invariant measures on them is twofold. First, recently Mai and Shi provided in \cite{MS17} very useful characterization of quasi-graphs spaces, showing that to some extent they are similar to graphs. While full characterization of $\omega$-limit sets on these spaces can be hard, characterization of invariant measures seems accessible. On dendrites, many new phenomena beyond dynamics possible on interval have been observed and many tools and characterizations were developed (e.g. see \cite{Ghasen}, \cite{Kocan} or \cite{Hoehn} to mention only a few).
Second, recent results show that M\"obius Disjointness Conjecture
holds (in the class of homeomorphism), provided that all invariant measures have discrete spectrum
(see \cite[Theorem 1.2]{HWY17}; cf. also \cite{HWZ16} and \cite{FKL17}; we refer the reader to a recent survey \cite{FKL17} on progress towards answering the conjecture).
%In \cite{S09}, Sarnak proposed the following conjecture (where $\mu$ is  the M\"obius function):
%
%The \textit{M\"obius function} $\mu\colon \N\to \{-1,0,1\}$  is defined as follows:
%$\mu(1)=1$, $\mu(n) = (-1)^k$ when $n$ is a product of $k$ distinct primes and $\mu(n)=0$
%otherwise.
%The well-known \emph{M\"obius Randomness Law} in \cite[Section 13.1]{IK04} speculates that summing the M\"obius sequence against any reasonable sequence would lead to significant cancellations.
%In \cite{S09}, Sarnak proposed a mathematically precise formulation of  ``reasonable sequence'' by considering bounded sequences arising from dynamical systems with zero topological entropy. In this way he reformulated the above mentioned law in terms of the following conjecture.
%\begin{conjecure}%[M\"obius Disjointness Conjecture]
%	The M\"obius function $\mu(n)$ is linearly
%	disjoint from any dynamical system $(X,T)$ with zero topological entropy, that is,
%	\[\text{for each}\ x\in X\ \text{and any}\ \varphi\in C (X, \mathbb{C}),\ \ \lim_{N\to\infty}\frac{1}{N}\sum_{n=1}^N \mu (n)\varphi(T^n(x))=0.\]
%\end{conjecure}
% by $C (X, \mathbb{C})$ the set of all continuous $\mathbb{C}$-valued functions over $X$.
%Though this conjecture remains still open,
%many interesting cases have been verified.
%While this condition is elegant, it applicability in practice is quite limited, since beyond simplest cases not much is known about spectrum of invariant measures of maps.
%the M\"obius Disjointness Conjecture holds in the class of these continua.
Many other earlier works prove with more direct arguments that for zero topological entropy maps on: interval \cite{K15}, circle \cite{Dav} or even topological graphs \cite{LOYZ17} (see also \cite{FJ15}) the conjecture holds. While for ergodic measures of these maps discrete spectrum can be deduced from classical characterizations of dynamics of topological graph maps due to Blokh \cite{Blokh1}, the structure of the simplex of all invariant measures seems complicated,
since clearly there may be uncountably many ergodic measures (e.g. see Example~\ref{ex:24}). In the authors opinion, it is natural to expect that all invariant measures have discrete spectrum in zero entropy graph maps, however up to their knowledge there is no such result in the literature.

Motivated by the above results, we state the main question for the paper:
\begin{que}\label{q:1}
What one-dimension continua $X$ have the property that if $(X,f)$ is dynamical system with zero topological entropy, then every invariant measure of $(X,f)$ has discrete spectrum?
\end{que}
In the above question we ask about connected spaces, because on the Cantor set there are numerous examples of dynamical systems with zero topological entropy, but with ergodic measures with continuous spectrum (e.g. celebrated Jewett-Krieger theorem provides a machinery for such examples \cite{Denker}). Also in dimension two it is relatively easy to provide an example which motivates us to work in dimension one (e.g. see Remark~\ref{torus}) ; by the same example discrete spectrum of all ergodic measures does not suffice to ensure discrete spectrum of all invariant measures.

To deal with Question~\ref{q:1}  we start with a class of continua called quasi-graphs as introduced in \cite{MS17}, which are to some extent similar to topological graphs (in the same way as the Warsaw circle is ``similar'' to the unit circle; they seem to be a special case of a slightly more general notion of generalized
$\sin(1/x)$-type continua of \cite{Hoehn}).
While it is quite easy to see that the structure of $\omega$-limit sets of maps on these spaces can be much richer
than was possible on topological graphs (recall from \cite{BlokhSurv}, that we have full characterization of admissible types of $\omega$-limit sets for topological graph maps), still some similarities exist.
First of all, we are able to characterize all invariant measures of quasi-graph maps as convex combinations of finitely many invariant measures for some topological graph maps. We show that every ergodic invariant measure of a map on a
quasi-graph is isomorphic to an ergodic invariant measure of a map on some topological graph and we also obtain an analog of Llibre-Misiurewicz result relating positive topological entropy
with existence of topological horseshoes.
%Furthermore, in the case of zero topological entropy, every invariant measure of a quasi-graph map has discrete spectrum,
%which by \cite[Theorem 1.2]{HWY17} (see also \cite{HWZ16} and \cite{FKL17}) the M\"obius Disjointness Conjecture holds in the class of these continua.

A partial answer to Question~\ref{q:1} is provided by the following theorem.
\begin{thm}\label{thm:quasi-graph-dis-spect}
	Let $X$ be a quasi-graph and let $f\colon X\to X$ be a continuous map with zero topological entropy. Then every invariant measure of $(X, f)$ has discrete spectrum.
\end{thm}

Dendrites are another important class of one-dimensional continua,
and we refer the reader to the textbook \cite{Nadler} for their basic properties.
Recall that a \textit{dendrite} is any locally connected continuum containing no simple closed curve.
Recently in \cite{Marzougi} the authors showed that
the M\"obius function is linearly disjoint from all monotone local dendrite maps and any dynamical system on a dendrite $X$ with zero topological entropy provided that
the set $End(X)$ of endpoints of this dendrite is closed with finitely many accumulation points.
%It was shown recently
%that every group action (in particular, every homeomorphism) on a dendrite is tame (cf. \cite[Theorem 1.1]{Glasner}), and then by \cite[Theorem 1.8]{HWY17} the M\"obius function is linearly
%disjoint from it.
%It is interesting to know that whether any invariant measure of  those dendrite maps has discrete spectrum.
So it is another candidate for the positive answer in Question~\ref{q:1}.
Unfortunately, we are only able to provide the following partial answer to the question.
\begin{thm}\label{thm:dendrite-finite}
	Let $X$ be a dendrite such that
	$End(X)$ is a closed set having finitely many accumulation points and
	let $f\colon X\to X$ be a continuous map with zero topological entropy.
	Then for any point $x\in X$,
	every invariant measure supported on $\overline{(\orb_f(x)},f)$ has discrete spectrum.
\end{thm}

In fact, we believe that the following question should have a positive answer,
but so far we do not have formal argument to prove it.
\begin{que}\label{que}
	Assume that $f\colon X\to X$ is a continuous map with zero topological entropy acting on a dendrite $X$ with $End(X)$ countable. Does every invariant measure of $(X, f)$ have discrete spectrum?
\end{que}

When it comes to dendrites with uncountably many endpoints,
the situation becomes different.
Recall that the \textit{Gehman dendrite} is the topologically unique dendrite whose set of all endpoints is homeomorphic to the Cantor set and whose branching points are all of order three (e.g., see \cite{Nadler}).
It is well known that, the one-sided full shift over $2$ symbols
can be embedded as a subsystem of a dynamical system on the Gehman dendrite (e.g., see \cite[Example~6]{Kocan}).
We can modify this construction to show the following interesting fact.

\begin{thm}\label{thm:Gehman-dendrite}
Every one-sided subshift $(V,\sigma)$
can be embedded to a dynamical system on the Gehman dendrite $(D,f)$
%with
%the same topological entropy.	
in such a way that each invariant measure of $(D,f)$ %$\mu\in M_f(D)$ 
is a convex combination of trivial measure and some invariant measure of $(V,\sigma)$.%$\nu \in M_\sigma(V)$.
\end{thm}

From the above theorem, it is quite obvious that the Gehman dendrite can support zero topological entropy maps with invariant measure beyond the class of these with discrete spectrum.
Namely, for every ergodic invariant measure $\mu$ of finite entropy, there exist a map $f$ on Gehman dendrite $X$, such that the set of ergodic measures of $(C,f)$%$M_f^e(X)$ 
consist of $\mu$ and trivial measure.
Furthermore, it is well known that if a dendrite has an uncountable set of endpoints then it contains a copy of the Gehman dendrite (e.g., see \cite[Corollary~2]{Nikiel}) and if $X\subset Y$ are dendrites, then there exists a standard retraction $\pi \colon Y\to X$ (so-called first point map). Then any continuous map $f\colon X\to X$ defines by $g=f\circ \pi$ a continuous map $g\colon Y\to Y$ with $g|_X=f$.
Therefore in Question~\ref{que} we cannot drop assumption that the set of endpoints is at most countable.

%In \cite{Boyle02} Boyle \textit{et al.} showed that
%every dynamical system with zero topological entropy
%is a factor of a two-sided subshift with zero topological entropy.
%%Combing this result with standard manipulations such as inverse limits and factors, together with Theorem~\ref{thm:Gehman-dendrite}
%%we see the following (note that subshift in Theorem~\ref{thm:Gehman-dendrite} is one-sided; the details are left to the reader).
%Combing this
%result with standard manipulations such as inverse limits and factors, together with
As we will see, Theorem \ref{thm:Gehman-dendrite} has the following consequence, which shows that the dynamics on some
one-dimensional continua, such as the Gehman dendrite, are sufficient to consider M\"obius
Disjointness Conjecture.
Recall, that 
%the \textit{M\"obius function} $\mu\colon \N\to \{-1,0,1\}$  is defined as follows:
%$\mu(1)=1$, $\mu(n) = (-1)^k$ when $n$ is a product of $k$ distinct primes and $\mu(n)=0$
%otherwise; M\"obius Disjointness Conjecture (see \cite{S09} ) 
the conjecture states that the M\"obius function $\mu(n)$ is linearly
disjoint from any dynamical system $(X,T)$ with zero topological entropy, that is,
\[\text{for each}\ x\in X\ \text{and any}\ \varphi\in C (X, \mathbb{C}),\ \ \lim_{N\to\infty}\frac{1}{N}\sum_{n=1}^N \mu (n)\varphi(T^n(x))=0.\]
where $C (X, \mathbb{C})$ denotes the set of all continuous $\mathbb{C}$-valued functions over $X$.

\begin{cor}\label{cor:factors}
	The M\"obius function is linearly
	disjoint from all dynamical systems with zero topological entropy
	if and only if
	it is so  for all surjective dynamical systems over the Gehman dendrite with zero topological entropy.
\end{cor}
This provides yet another evidence that one-dimensional dynamics is not much simpler than dynamics in full generality.

The paper is organized as follows. In Section 2 we explore maximal $\omega$-limit sets of topological graph maps with zero topological entropy and show that every invariant measure for map with zero topological entropy acting on a topological graph map has discrete spectrum.
 In Section 3 we characterize all
invariant measures of quasi-graph maps as convex combinations of finitely many invariant measures for some topological graph maps. Consequently, we show that every ergodic invariant measure of a
quasi-graph map is isomorphic to an ergodic invariant measure of some topological graph map. We also obtain an analog of Llibre-Misiurewicz result relating positive topological entropy
with existence of topological horseshoes.
Finally, in Section \ref{main} we study dynamics on dendrite maps and present proofs of Theorems~\ref{thm:dendrite-finite}
and \ref{thm:Gehman-dendrite} and Corollary~\ref{cor:factors}.

\section{Graph maps with zero topological entropy}

In this section, we discuss invariant measures and discrete spectrum for general dynamical systems and then
study the structure of $\omega$-limits of topological graph maps
and show that if a topological graph map has zero topological entropy
then every invariant measure has discrete spectrum.

Recall that a topological space $X$ is a \textit{continuum} if it is a compact, connected metric space.
An \textit{arc} is a continuum homeomorphic to the closed interval $[0,1]$. A \textit{topological graph}, or just \textit{graph} for short, is a continuum which is a union of finitely many arcs,
any two of which are either
disjoint or intersect in at most one common endpoint.
We say that a graph $S$ is an \emph{$n$-star with center $v\in S$}
if there is a continuous injection $\varphi\colon S\to \C$ such that $\varphi(v)=0$
and $\varphi(S)=\{r \exp(\frac{2 k \pi i}{n}): r\in[0,1],\ k=1,2,\dotsc,n\}$.

Let $X$ be a compact arcwise connected metric space and $v\in X$.
The \emph{valence} of $v$ in $X$, denoted by $val(v)$,
is the number
$$\sup\{n\in\N\colon\text{there exists an }n\text{-star with center }
v\text{ contained in }X\}.
$$
Note that the valence of $v$ may be $\infty$.
The point $v$ is called an \emph{endpoint} of $X$ if $val(x)=1$,
and a \emph{branching point} of $X$ if $val(x)\geq 3$.
The collections of all endpoints and branching points of $X$ are denoted by
$End(X)$ and $Br(X)$, respectively.

\subsection{Invariant measures and discrete spectrum}
Let $(X,f)$ be a dynamical system and $x\in X$.
The \textit{orbit of $x$}, denoted by $\orb_f(x)$, is the set $\{f^nx\colon n\geq 0\}$, and the \textit{$\omega$-limit set of $x$}, denoted by $\omega_f(x)$, is defined as $\bigcap_{n\ge 0} \overline{\{f^m x\colon  m\ge n\}}$. It is easy to check that $\omega_f(x)$ is closed and strongly $f$-invariant, i.e., $f (\omega_f(x))= \omega_f(x)$.

Let $(X, f)$ be a dynamical system.
The set of all \textit{Borel probability measures} over $X$ is denoted by $M(X)$, and $M_f(X)\subset M(X)$ denotes the set of all \textit{invariant} elements of $M(X)$.
The set of all \textit{ergodic} measures in $M_f(X)$ is denoted by $M_f^e(X)$.
It is well known that $M(X)$ endowed with weak-* topology is a compact metric space and that
$M_f(X)$ is its closed subset.
We say that $\mu\in M_f(X)$ has \textit{discrete spectrum}, if the linear span of the eigenfunctions of $U_f$ in $L^2_\mu(X)$ is dense in $L^2_\mu(X)$, where as usual $U_f$ denotes the \textit{Koopman operator}: $U_f(\varphi)=\varphi\circ f$ for every $\varphi\in L^2_\mu (X)$.
By a classical result by Ku\v{s}hnirenko, an invariant measure
has discrete spectrum if and only if it has zero measure-theoretic
sequence entropy \cite{K67}.
We refer the reader to  \cite{Downar} and \cite{W82} as standard monographs on ergodic theory and entropy.

\begin{rem}\label{rem:simpleds}
	Let $\mu$ be an invariant measure for a dynamical system $(X, f)$. Assume that $X=\bigcup_{n=1}^\infty X_n$, where each $X_n$
	is a Borel set and each $X_n\cap \text{supp} (\mu)$ is positively $f$-invariant (i.e., $f (X_n\cap \text{supp} (\mu))\subset X_n\cap \text{supp} (\mu)$; where as usual $\text{supp} (\mu)$ denotes the \textit{support} of the measure $\mu$). Then $\mu$ has discrete spectrum, provided that the normalized invariant measure of each measure $\mu|_{X_n}$ has discrete spectrum.
\end{rem}

It is shown in \cite{C07} that every invariant measure for
an interval map with zero topological entropy has zero measure-theoretic sequence entropy.
We shall generalize this result in next subsection to graph maps with
zero topological entropy, see Theorem~\ref{thm:graph-map-dis-spect}.
In \cite{LOYZ17}, we have shown that
every ergodic invariant measure on a graph map
with zero topological entropy has discrete spectrum.
The following example shows that  in general situation
it is not enough to ensure that every  invariant measure
has discrete spectrum.

\begin{rem}\label{torus}
	The Lebesgue measure is an invariant measure of
	the map $(x,y)\mapsto (x,y+x)$ on torus which does not have
	discrete spectrum (see, e.g., \cite{A51} or \cite{K67}), while all ergodic invariant measures come from rotations of the circle, therefore all of them have discrete spectrum.
\end{rem}

It is well known that if a dynamical system has only countably many
ergodic invariant measures and each of them has discrete spectrum
then any invariant measure also has discrete spectrum.
This statement can be generalized slightly.

\begin{thm}\label{thm:discrete}
	Let $(X,f)$ be a dynamical system and $\mu\in M_f(X)$.
	Let $\mu=\int_{M_f^e(X)} \nu d\rho(\nu) $ be the ergodic decomposition of $\mu$.
	Assume that there are at most countably many non-atomic ergodic invariant measures in the decomposition (i.e., we can rewrite $\mu$ as $\mu=\int_{Q} \nu d\rho(\nu)$ with set $Q\subset M_f^e(X)$ containing at most countably many non-atomic measures) and each of them has discrete spectrum. Then $\mu$ also has discrete spectrum.
\end{thm}

\begin{proof}
Denote $X_n=\{ x\in X : f^n(x)=x \text{ and }f^i(x)\neq x \text{ for all }i=1,2,\ldots,n-1\}$ for each $n\in \mathbb{N}$ and $X_\infty=X\setminus \bigcup_{n=1}^\infty X_n$.
Clearly sets $X_n, n\in \mathbb{N}\cup \{\infty\}$ are pairwise disjoint $f$-invariant Borel subsets and then (the formula makes sense, no matter if $\mu|_{X_n}, n\in \mathbb{N}\cup \{\infty\}$ is trivial or not)
$$
L^2(\mu)=\bigoplus_{n\in \mathbb{N}} L^2(\mu|_{X_n})\oplus L^2(\mu|_{X_\infty}).
$$

Now applying Remark~\ref{rem:simpleds} it is enough to show that after restriction to any of these sets $\mu$ has discrete spectrum.
For each $n\in \mathbb{N}$, since $f^n|_{X_n}$ is the identity, $U_f^n$ is the identity as well on $L^2(\mu|_{X_n})$, and then the spectrum of $U_f$ restricted onto $X_n$ consists of exactly $n$-th roots of $1$, thus the normalized invariant measure of $\mu|_{X_n}$ has discrete spectrum.

But by the assumption $X_\infty$ has positive measure for at most countably many ergodic invariant measures from $Q$, all of which are pairwise singular, so similarly we have
$$
L^2(\mu|_{X_\infty})=\bigoplus_{\nu\in Q'} L^2(\mu|_{X_\nu}),
$$
where $Q'$ denotes the countable set of all non-atomic ergodic invariant measures of $(X, f)$ from $Q$ and for each $\nu\in Q'$ the set $X_\nu$ is a Borel set such that $\nu(X_\nu)=1$ and $\eta(X_\nu)=0$ for all $\eta\in M^e_\infty(f)\setminus \{\nu\}$,
e.g., we can take $X_\nu$ to be the set of all generic points of $\nu$.
Combining again Remark~\ref{rem:simpleds} with the assumption one has readily that the normalized invariant measure of $\mu|_{X_\infty}$ has discrete spectrum. This finishes the proof.	
\end{proof}

\begin{figure}
	\begin{subfigure}{0.49\textwidth}
		\begin{center}
			\includegraphics[width=0.7\textwidth]{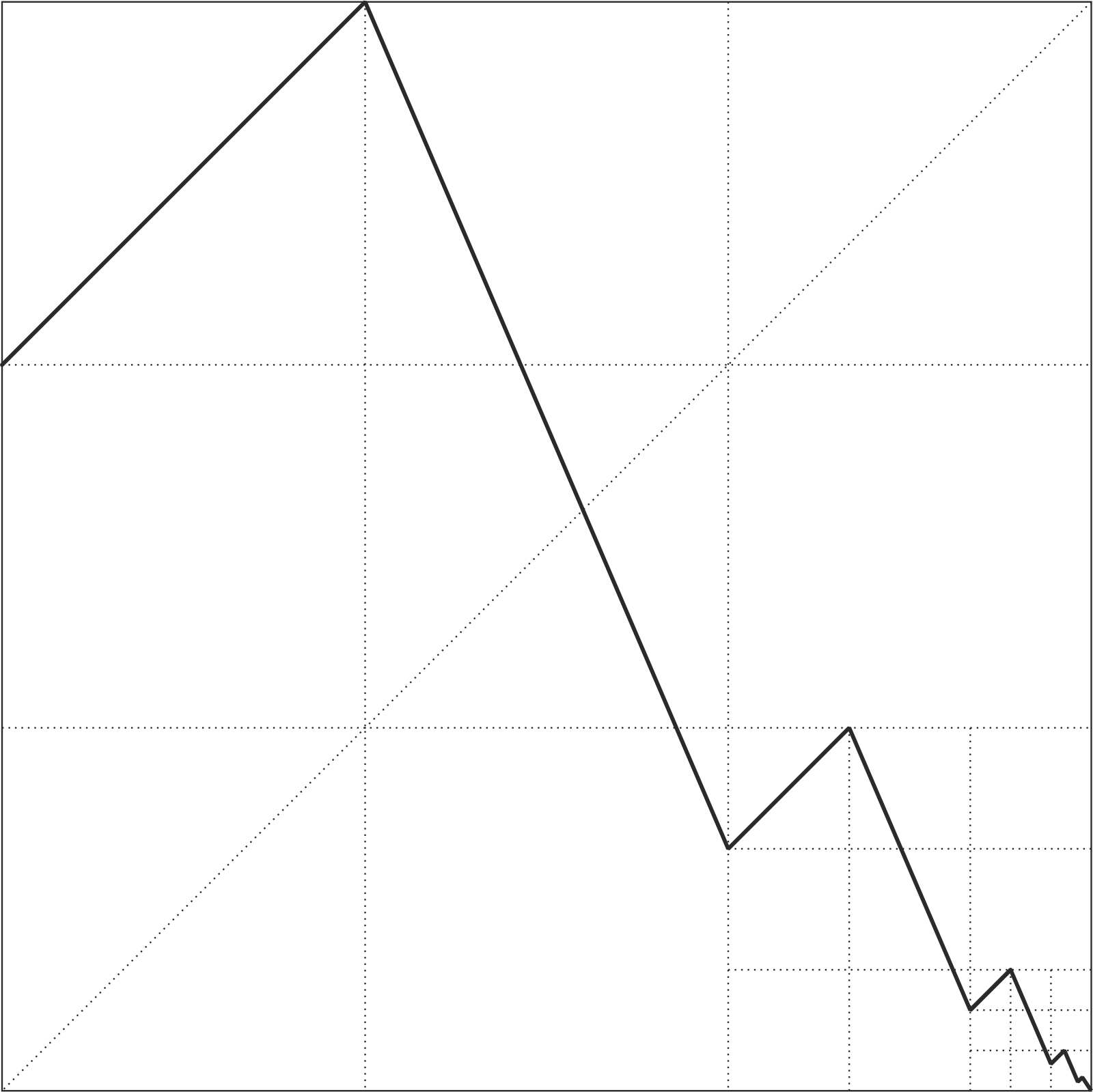}
			\caption{Classical Delahaye's example }\label{fig:mapf}
		\end{center}
	\end{subfigure}
	\begin{subfigure}{0.49\textwidth}
		\begin{center}
			\includegraphics[width=0.7\textwidth]{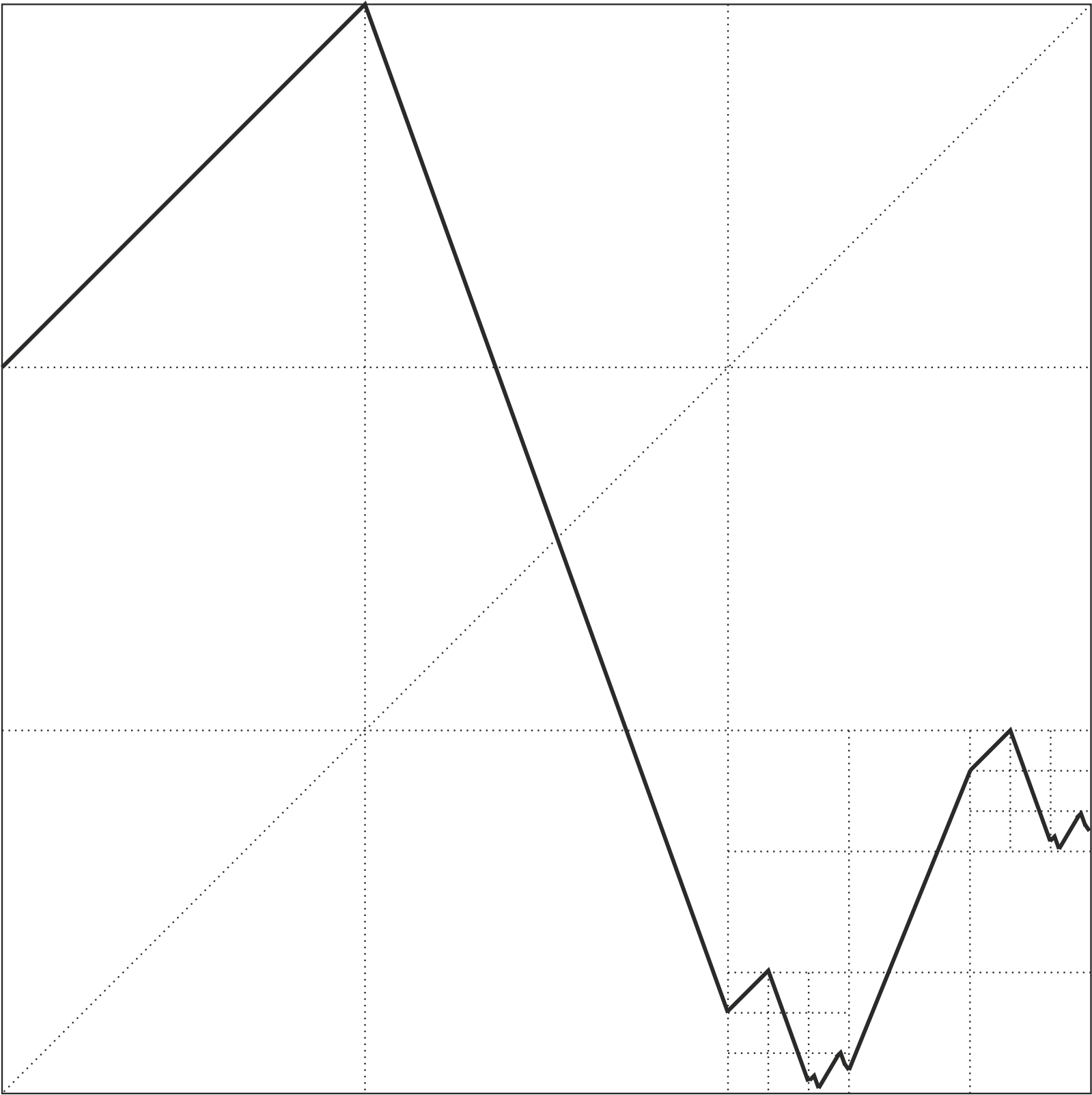}
			\caption{Modification of (a) to obtain uncountably many odometers}\label{fig:mapf2}
		\end{center}
	\end{subfigure}
	\caption{\small Zero topological entropy maps with odometers as maximal $\omega$-limit sets}\label{fig:uncountable}
\end{figure}

The following example reveals that there exists an interval map with
zero topological entropy which has uncountably many non-atomic ergodic measures.
So we cannot apply Theorem~\ref{thm:discrete} directly.

\begin{exam}\label{ex:24}
	Figure~\ref{fig:mapf} presents an example of Delahaye from \cite{Dela} with odometer as a maximal $\omega$-limit set.
	Its simple modification Figure~\ref{fig:mapf2} leads to a map with uncountably many maximal $\omega$-limit sets which are odometers, and consequently,
	uncountably many non-atomic ergodic measures. The idea here is very simple. In original Delahaye's example, the interval $[0,1/3]$ is $2$-periodic with image $[2/3,1]$
	and over $[2/3,1]$ the same scheme is repeated for $f^2$. In modified example, again $[0,1/3]$ is periodic, but now it contains two $2$-periodic subintervals with disjoint orbits,
	namely $[2/3,7/9]$ and $[8/9,1]$. Since at each step of construction $2^n$-periodic interval splits into two periodic intervals with disjoint orbits, at the end we will have uncountably many invariant sets, which by the construction are odometers.
\end{exam}

Let $(X,f)$ be a dynamical system.
Recall that a pair $(x,y)$ is \textit{proximal} if
\[
	\liminf_{n\to\infty}d(f^n(x),f^n(y))=0,
\]
\textit{asymptotic} if
\[
	\lim_{n\to\infty}d(f^n(x),f^n(y))=0,
\]
and a \textit{Li-Yorke pair}  (or \emph{scrambled pair})
if it is proximal but not asymptotic.
Given two dynamical systems $(X,f)$ and $(Y, g)$, by a \textit{factor map} we mean $\pi\colon (X, f)\rightarrow (Y, g)$ where $\pi\colon X\rightarrow Y$ is a continuous surjection satisfying $\pi\circ f= g\circ \pi$; in this case, $(X, f)$ is called an \textit{extension} of $(Y, g)$ and $(Y, g)$ is called a \textit{factor} of $(X, f)$.

We will use the following result which is essentially contained in the proof of
\cite[Theorem 3.8]{LTY15}
(see also proofs of \cite[Proposition 2.5]{DG16} and
\cite[Theorem 4.4]{HLSY03}).

\begin{lem}\label{lem:factor-maps}
	Let $\pi\colon (X,f)\to (Y,g)$ be a factor map, and set $\Delta_X= \{(x, x): x\in X\}$ and $R_\pi=\{(x_1,x_2)\in X^2\colon \pi(x_1)=\pi(x_2)\}$.
	Assume that $\Delta_X$ has full measure for each invariant measure of $(R_\pi,f\times f)$.
	Then every invariant measure of $(X,f)$
	is measure-theoretically isomorphic to some invariant measure of $(Y,g)$.
\end{lem}

\subsection{The structure of $\omega$-limits of graph maps with zero topological entropy}
In this subsection we fix a topological graph $G$ and
a continuous map $f\colon G\to G$ with zero topological entropy.

A subgraph $K$ of $G$ is called \textit{periodic}
if there is a positive integer $k$ such that $K$,
$f(K)$, $\dotsc$, $f^{k-1}(K)$
are pairwise disjoint and $f^k(K) = K $. In such a case,
$k$ is called the period of $K$ and $\bigcup_{i=0}^{k-1}f^i(K)$
is called  a \emph{cycle of graphs}.
For an infinite $\omega$-limit set $\omega_f(x)$ we let
$$
C(x)=\{ X \colon X\subset G \text{ is a cycle of graphs and }\omega_f(x)\subset X\}.
$$
By \cite[Lemma1]{Blokh1} (see also \cite[Lemma~10~(i)]{RS14}) the family $C(x)$ is never empty.
If periods of cycles of graphs in $C(x)$ are unbounded then $\omega_f(x)$ is called \textit{solenoid} (\textit{for $(G, f)$}).
Below we recall \cite[Theorem~1]{Blokh1} which is an effective tool when dealing with solenoids (see also \cite[Lemma~11]{RS14}).

\begin{lem}\label{lem:solenoid}
	Let $f\colon G\to G$ be a graph map with zero topological entropy.
	Let $\omega_f (x)$ be a solenoid. Then there exists a sequence of cycles of graphs $(X_n)_{n\geq 1}$ with strictly increasing periods $k_n$ such that
\begin{enumerate}
	\item for every $n\ge 1$, $k_{n+1}$ is a multiple of $k_n$;
	\item for every $n\geq 1$, $X_{n+1}\subset X_n$;
	\item for every $n\ge 1$, each connected component of $X_n$
	contains the same number of connected components of $X_{n+1}$; and
	\item  $\omega_f(x)\subset \bigcap_{n\ge 1} X_n$.
	\end{enumerate}
 Furthermore, $\omega_f(x)$ does not contain any periodic points.
\end{lem}

Note that the set of all $\omega$-limit sets of a graph map is closed under Hausdorff metric~\cite{MS07}, and then each $\omega$-limit set is contained in
a maximal one by Zorn Lemma.
The following three lemmas are implicitly contained in~\cite{Blokh1},
here we provide proofs for completeness.

\begin{lem} \label{disjoint}
	Let $f\colon G\to G$ be a graph map with zero topological entropy.
Any two different solenoids which are maximal $\omega$-limit sets are disjoint.
\end{lem}
\begin{proof}
Assume that $\omega_f(x)$, $\omega_f(y)$ are two solenoids which are both maximal $\omega$-limit sets.
As $G$ has only finitely many branching points,
by Lemma~\ref{lem:solenoid} there exists a positive integer $n$
such that $k_n$ is larger than the number of branching points of $G$
and so there is also a connected  component $I$ of $X_n$ which does not contain any branching points of $G$.
Then $I$ is an arc in $G$ and $f^{k_n}(I)=I$.
It is clear that $\omega_f(x)\cap I$ is uncountable,
so $f^s(x)\in I$ for some $s>0$.
Then $\omega_{f^{k_n}}(f^s(x))\subset I$ is  a solenoid for $(I, f^{k_n})$.
Furthermore, it is not hard to see that $\omega_{f^{k_n}}(f^s(x))$ is a maximal $\omega$-limit set for $(I, f^{k_n})$. Namely, if $\omega_{f^{k_n}}(\xi)\supset \omega_{f^{k_n}}(f^s(x))$ is another $\omega$-limit set for $(I, f^{k_n})$, then $\omega_{f}(\xi)\supset \omega_{f}(x)$ is an $\omega$-limit set for $(G, f)$ and hence $\omega_{f}(\xi)= \omega_{f}(x)$, so finally one has $\omega_{f^{k_n}}(\xi)= \omega_{f^{k_n}}(f^s(x))$, since $\omega_{f^{k_n}}(\xi)= \omega_{f}(\xi)\cap I$ and $\omega_{f^{k_n}}(f^s(x))= \omega_{f}(x)\cap I$ by the above construction.
	
Now suppose that there is $z\in \omega_f(y)\cap \omega_f(x)$. By invariance of $\omega_f(y)\cap \omega_f(x)$ and its periodic structure provided by Lemma~\ref{lem:solenoid} we may assume that $z$ belongs to the interior of $I$. Similar to the above arguments, there is $r>0$ such that $f^r(y)\in I$,
$\omega_{f^k}(f^r(y))$ is a maximal $\omega$-limit set for $(I, f^{k_n})$ and $\omega_{f^{k_n}}(f^r(y))= \omega_{f}(y)\cap I$. In particular, $z$ belongs to the intersection of $\omega_{f^{k_n}}(f^s(x))$ and $\omega_{f^{k_n}}(f^r(y))$. But now we are on the interval $I$ and
clearly $(I, f^{k_n})$ has zero topological entropy, so by \cite[Lemma~3.4]{L11} one has $\omega_{f^{k_n}}(f^s(x))= \omega_{f^{k_n}}(f^r(y))$, which leads to $\omega_{f}(x)= \omega_{f}(y)$. This finishes the proof.
\end{proof}

\begin{lem}\label{lem:max}
	Let $f\colon G\to G$ be a graph map with zero topological entropy.
Let $\omega_f(x)$ be a solenoid, and assume $\omega_f(x)\subset \bigcap_{n\ge 1} X_n$ where $(X_n)_{n\geq 1}$ is a sequence of cycles of graphs with increasing periods provided by Lemma~\ref{lem:solenoid}.
Then:
\begin{enumerate}
  \item the set $P(x):=\bigcap_{n\ge 1} X_n$ is closed, strongly invariant and does not contain any periodic points;
  \item $P(x)$ has uncountably many
  connected components and at most countably many of them can be
  non-degenerate;
  \item each connected component of $P(x)$ intersects $\omega_f(x)$;
  \item\label{lem:max:4} if $J$ is a non-degenerated connected component of $P(x)$,
  then %the interior of
  $J$ is wandering, i.e., %$f^i(int(J))\cap int(J)=\emptyset$ for all $i\geq 1$;
  $f^i(J)\cap J=\emptyset$ for all $i\geq 1$;
  \item $P(x)$ depends only on $\omega_f(x)$, not on the choice of $(X_n)_{n\geq 1}$;
  \item for any $y\in P(x)$,
  $\omega_f(y)\setminus\omega_f(x)$ is at most countable;
  \item for every $z\in G$, $\omega_f(z)\cap \omega_f(x)\neq\emptyset$ implies that $\omega_f(z)$ is also a solenoid contained in $P(x)$ and
    $\omega_f(z)\setminus\omega_f(x)$ is at most countable; and

    \item there exists a unique maximal $\omega$-limit set containing $\omega_f(x)$ and this maximal $\omega$-limit is a solenoid which is contained in $P(x)$.
\end{enumerate}
\end{lem}
\begin{proof}
(1)--(4) can be easily obtained from Lemma~\ref{lem:solenoid} and the construction of $P(x)$.

(5) Let $(Y_n)_{n\ge 1}$ be another sequence of cycles of graphs as in Lemma~\ref{lem:solenoid}.
Note that $\bigcap_{n\ge 1} Y_n$ also satisfies (1)--(4).
In particular, $\bigcap_{n\ge 1} Y_n$ has uncountably many degenerate
components and each of them is contained in $\omega_f(x)$.
For every sufficiently large integer $m$
there is a connected component $J_m$ of $X_m$ which
does not contain branching points of $G$ and such that $J_m$ contains at least three degenerate components of $\bigcap_{n\ge 1} Y_n$.
But then by the assumption there exists a positive integer $s$ and a connected component $I_s$ of $Y_s$ such that $I_s\subset J_m$, which implies that $Y_s\subset X_m$ and hence $\bigcap_{n\ge 1} Y_n\subset \bigcap_{n\ge 1} X_n$. Finally by duality we obtain $\bigcap_{n\ge 1} Y_n= \bigcap_{n\ge 1} X_n$.
	
(6) Let $y\in P(x)$. Clearly $\omega_f(y)\subset P(x)$. By (1), $\omega_f(y)$ does not contain any periodic points, therefore $\omega_f(y)$ is uncountable.
By (4), for every non-degenerated connected component $J$ of $P(x)$,
$J\cap \omega_f(y)$ is contained in the boundary of $J$ and then is finite. So there are at most countable many points in $\omega_f(y)$
which are in the non-degenerated connected components $J$ of $P(x)$. By (3), other points in $\omega_f(y)$ must be contained in $ \omega_f(x)$. Then $\omega_f(y)\setminus\omega_f(x)$ is at most countable.

(7) First note that $\omega_f(z)\cap \omega_f(x)$ is uncountable.
For any positive integer $n$,
as $\omega_f(x)\subset X_n$ and $X_n$ has only finite boundary points,
there exists a positive integer $i$ such that $f^i(z)\in X_n$
and then $\omega_f(z)\subset X_n$.
Then $\omega_f(z)$ is also a solenoid and $\omega_f(z)\subset P(x)$.
Similar to the proof of (6), $\omega_f(z)\setminus\omega_f(x)$ is at most countable.

(8) Let $\omega_f(y)$ be a maximal $\omega$-limit set containing $\omega_f(x)$.
In particular, $\omega_f(y)\cap \omega_f(x)\neq\emptyset$, and hence $\omega_f(y)$ must be a solenoid contained in $P(x)$ by (7).
If $\omega_f(z)$ is another maximal $\omega$-limit set containing $\omega_f(x)$, then clearly $\omega_f(y)\cap\omega_f(z)\supset \omega_f(x)\neq\emptyset$.
Then $\omega_f(y)=\omega_f(z)$ follows from Lemma~\ref{disjoint}.
\end{proof}

\begin{lem} \label{disjoint-P}
	Let $f\colon G\to G$ be a graph map with zero topological entropy.
	Let $\omega_f(x)$ and $\omega_f(y)$ be two different solenoids which are maximal
	$\omega$-limit sets and
	assume that $P(x)$ and $P(y)$ are as in Lemma~\ref{lem:max}.
	Then $P(x)$ and $P(y)$ are disjoint.
\end{lem}
\begin{proof}
If $P(x)\cap P(y)\neq\emptyset$, pick $z\in P(x)\cap P(y)$.
By Lemma~\ref{lem:max}, $\omega_f(z)$ is uncountable,
and both $\omega_f(z)\setminus\omega_f(x)$ and $\omega_f(z)\setminus\omega_f(y)$ are at most countable.
Then $\omega_f(x)\cap\omega_f(y)\neq\emptyset$, and hence $\omega_f(x)= \omega_f(y)$ by Lemma~\ref{disjoint}. A contradiction. Thus, $P(x)\cap P(y)=\emptyset$.
\end{proof}

Thus we can obtain the following result which will be used later.

\begin{prop}\label{prop:collapse}
	Let $f\colon G\to G$ be a graph map with zero topological entropy.
There exists a continuous map $g$ acting on a graph $Y$ without Li-Yorke pairs and a factor map $\pi \colon (G,f)\to (Y,g)$
such that the pair $(p,q)$ is asymptotic whenever
 $p,q\in \pi^{-1}(y)$ for some $y\in Y$.
\end{prop}
\begin{proof}
For every maximal $\omega$-limit set which is a solenoid, denote by $P$ the unique set associated to it as in Lemma~\ref{lem:max}, by intersecting periodic graphs provided by Lemma~\ref{lem:solenoid}.
Note that there are at most countably many sets
$P$ which have non-degenerate connected components,
as we are working on a graph $G$ and sets $P$ are by Lemma~\ref{disjoint-P} either disjoint or equal.
Denote by $(P_n)_{n\in \Lambda}$ the sequence (finite or not) of these sets, where the index set $\Lambda$ is at most countable,
%By Lemma~\ref{disjoint-P},
such that $P_n\cap P_m=\emptyset$ whenever $m\neq n$.
Let $R$ be the relation on $G\times G$ given by $x\sim y$ if and only if $x=y$ or there exists an $n\in \Lambda$ and a connected component $C$ of $P_n$ such that $x,y\in C$. In particular, if $\Lambda= \emptyset$ then $(x, y)\in R$ if and only if $x= y$. By definition $R$ is positively $f\times f$-invariant, i.e., $(f\times f) (R)\subset R$.
Since for every $\eps>0$ there are at most finitely many  disjoint connected subsets of diameter at least $\eps$, we immediately obtain that the subset $R$ is closed.
It is also clear that equivalence classes of $R$ are connected (and so arc-wise connected).
We have readily that the quotient space $Y=G/_R$ is a topological graph.
	
We are going to check that the induced maps $g=f/_R$ and $\pi \colon (G,f)\ni x \mapsto [x]_R\in (Y,g)$ have desired properties.
First, let $y\in Y$ and fix any $p,q\in \pi^{-1}(y)$. If $p\neq q$ then %$p, q\in P_n$ for some $n\in \Lambda$.
%Note that
by the definition of $\pi$, points $p, q$ are contained in the same connected component of $P_n$ for some $n\in \Lambda$, so by Lemma~\ref{lem:max}\eqref{lem:max:4}
the pair $(p,q)$ is asymptotic.
Now it is sufficient to show that $(Y,g)$ contains no Li-Yorke pairs.
Repeating the argument from the proof of \cite[Theorem~3]{RS14},
we only need to check that for every solenoid $\omega_f(y)$ in $Y$,
$P(y)$ does not contain any non-degenerate connected components.
Pick a point $x\in G$ with $\pi(x)=y$.
Then $\pi(\omega_f(x))=\omega_f(y)$.
As $\pi$ is monotone (i.e., pre-images of connected sets are connected), $\omega_f(x)$ is also a solenoid and $\pi(P(x))=P(y)$.
By the construction of $Y$, we have  collapsed all the
non-degenerated connected components of $P(x)$,
so $P(y)$ does not contain any non-degenerate connected components. This completes the proof.
\end{proof}

Now we can prove the main result of this section.
\begin{thm}\label{thm:graph-map-dis-spect}
Let $f\colon G\to G$ be a graph map with zero topological entropy.
Then every invariant measure of $(G, f)$ has discrete spectrum.
\end{thm}
\begin{proof}
Let $\pi\colon (G,f)\to (Y,g)$ be the factor map provided by Proposition~\ref{prop:collapse}.
Since the graph map $(Y,g)$ does not contain Li-Yorke pairs,
by \cite[Theorem~1.5]{LOYZ17} it has zero topological sequence entropy.
Then each invariant measure has zero sequence entropy,
because metric sequence entropy cannot exceed the value of topological sequence entropy (see \cite[Theorem2.6]{HY09}).
By \cite[Theorem~4]{K67}, each invariant measure of $(Y,g)$ has
discrete spectrum.
Note that every pair in $R_\pi \setminus \Delta_G$ is asymptotic, so it is immediate to
see that $\Delta_G$ has full measure for each invariant measure of $(R_\pi, f\times f)$, since it contains all recurrent points of $f\times f$ in $R_\pi$.
By Lemma~\ref{lem:factor-maps},
every invariant measure of $(G,f)$
is measure-theoretically isomorphic to some invariant measure of $(Y,g)$ and hence also has discrete spectrum.
\end{proof}

\section{Dynamics on quasi-graphs} \label{quasi-graph}
In this section, we characterize all
invariant measures of quasi-graph maps as convex combinations of finitely many invariant measures for some graph maps. Consequently, we show that every invariant measure of a quasi-graph map with zero topological entropy has discrete spectrum
and every ergodic invariant measure of a
quasi-graph map is essentially an ergodic invariant measure of some graph map. We also obtain an analog of Llibre-Misiurewicz result relating positive topological entropy
with existence of topological horseshoes.

\subsection{Preliminaries on quasi-graph maps}
Let $X$ be a compact metric space and let $L$ be an arcwise connected subset of $X$. By convention, we denote $\R_+=[0,\infty)$.
If there exists a continuous bijection
$\varphi \colon \R_+ \to L$,
then we say that $L$ is a \textit{quasi-arc with the parameterization $\varphi$}.
The point $\varphi(0)$ is called an \emph{endpoint} of $L$ and
the \emph{$\omega$-limit set} of $L$  is the set
$\omega(L)=\bigcap_{m\geq 0} \overline{\varphi([m,\infty))}$.
A quasi-arc $L$ is called \textit{oscillatory}
if $\omega(L)$ contains more than one point.
Note that the endpoint and the $\omega$-limit set of $L$ are by the definition dependent on
the parameterization $\varphi$, however
if $L$ is an oscillatory quasi-arc then the endpoint is uniquely determined and
the parameterization is unique up to homeomorphism of $\R_+$ (cf. \cite[Propositions 2.17 and 2.20]{MS17}).

A \textit{quasi-graph} is a non-degenerate, compact, arcwise connected metric space $X$ satisfying that there is a positive integer $N$ such that $\overline{Y} \setminus Y$ has at most $N$ arcwise connected components for every arcwise connected subset $Y\subset X$.
The following fact
from \cite[Theorem~2.24]{MS17} is an important characterization of quasi-graphs.

\begin{thm}\label{thm:dec-qg}
	A continuum $X$ is a quasi-graph if and only if there is a graph $G$ and pairwise disjoint oscillatory quasi-arcs $L_1,\ldots , L_n$ (with $n\geq 0$) in $X$ such that:
	\begin{enumerate}
		\item $X=G\cup \bigcup_{j=1}^n L_j$ and $End(X)\cup Br(X)\subset G$,
		\item for each $1\leq i\leq n$, $L_i \cap G=\{a_i\}$ where $a_i$ is the endpoint of $L_i$,
		\item $\omega(L_i)\subset G\cup \bigcup_{j=1}^{i-1} L_j$ for each $1 \leq  i \leq n$, and
		\item if $\omega (L_i) \cap L_j\neq  \emptyset$ for some $1\leq i,j\leq n$, then $\omega (L_i) \supset L_j$.
	\end{enumerate}
\end{thm}

Note that in Theorem \ref{thm:dec-qg}, the case of $n=0$ is the simplest situation, when a quasi-graph is in fact a graph.

First we have the following useful observation on quasi-graphs.

\begin{prop}\label{prop:arcs-in-QG}
Let $X$ be a quasi-graph with $G$ and $L_1,\dotsc,L_n$ as in Theorem~\ref{thm:dec-qg}.
Then: for every two different points $a,b\in X$ there are only finitely many different
arcs in $X$ with endpoints $a$ and $b$,
furthermore, if $a$ and $b$ are in the same quasi-arc $L_i$
then there is a unique arc with endpoints $a$ and $b$.
\end{prop}
\begin{proof}
By Theorem~\ref{thm:dec-qg}, all the endpoints and branching points of $X$ are in $G$.
As $G$ is a graph, the sum of valences of all branching points is finite.
Then for every two different points $a,b\in X$ there are only finite many different
arcs in $X$ with endpoints $a$ and $b$.

Now assume that $a$ and $b$ are in the same quasi-arc $L_i$, and let $\varphi\colon \R_+\to L_i$ be a parameterization of $L_i$.
By Theorem~\ref{thm:dec-qg}, $\omega(L_i)\cap L_i=\emptyset$, and so
$\varphi$ is a homeomorphism.
Pick $s,t\in\R_+$ such that $\varphi(s)=a$ and $\varphi(t)=b$.
Without loss of generality, assume that $s<t$.
Then $\varphi|_{[s,t]}$ is an arc with $\varphi(s)=a$ and $\varphi(t)=b$.
Let $\alpha\colon [0,1]\to X$ be another arc (different from $\varphi$) with endpoints $a$ and $b$, and say $\alpha(0)=a$ and $\alpha(1)=b$.
As every point in $\varphi((s,t))$ has valence $2$,
$\alpha((0,1))\cap \varphi((s,t))=\emptyset$.
Then there exists $c\in (0,1)$ such that $\alpha((c,1))=\varphi((t,\infty))$.
But this implies that $\omega(L_i)=\{\alpha(c)\}$,
which is in contradiction to the assumption in Theorem~\ref{thm:dec-qg} that $L_i$ is an oscillatory quasi-arc.
\end{proof}

By the above, any non-degenerate arcwise connected closed set $H$ of a quasi-graph $G$ is again a quasi-graph, but the positive integer $N$ in the definition can increase as some path in $G$ may not belong to $H$.

\begin{prop}\label{prop:quasi-graph-onto}
Let $X$ be a quasi-graph and let $f\colon X\to X$ be a continuous map.
Then  $\bigcap_{n=0}^\infty f^n(X)$ is also a quasi-graph if it is non-degenerate.
\end{prop}
\begin{proof}
First note that $f(X)$ is a quasi-graph
since it is arcwise connected as an image of an arcwise connected set and consequently,
$f^n(X)$ is a quasi-graph  for every $n\geq 1$.
Let $X_0=\bigcap_{n=0}^\infty f^n(X)$.
It is sufficient to show that $X_0$ is arcwise connected.
Fix any two different points $p,q\in X_0$ (if they exist).
For every $n\geq 1$, as $f^n(X)$ is arcwise connected,
there exists an arc $J_n\subset X$ with endpoints $p,q$ such that $J_n\subset f^n(X)$.
By Proposition~\ref{prop:arcs-in-QG},
there are only finitely many different arcs in $X$ connecting $p$ and $q$.
Therefore, since $f^n(X)$ is a nested sequence,
there exists an  arc $J\subset X$ with endpoints $p,q$
such that $J\subset f^n(X)$ for all $n\geq 0$.
Then $J\subset X_0$, which implies that $X_0$ is arcwise connected.
\end{proof}

A non-oscillatory quasi-arc in a compact metric space $X$ is called
a \textit{$0$-order oscillatory quasi-arc}.
An oscillatory quasi-arc $L$ is called a \textit{$k$-order oscillatory quasi-arc}
for some $k>0$ if $\omega(L)$ contains at least one $(k-1)$-order oscillatory
quasi-arc, and $\omega(K)$ does not contain any $(k-1)$-order oscillatory quasi-arc
for every quasi-arc $K$ in $\omega (L)$.
It is not hard to see that the $\omega$-limit set $\omega(L)$ of an
oscillatory quasi-arc $L$ of order $k$
contains at least one oscillatory quasi-arc $K_i\subset \omega(L)$ of order $i$
for each $i=0,1,\ldots, k-1$, and does not contain any quasi-arc of order $n\geq k$.
The following lemma combines \cite[Lemma~3.1, Corollaries 3.2 and 3.3
and Proposition 3.4]{MS17}.

\begin{lem}\label{lem:img-qg}
Let $X$ be a quasi-graph and let $f\colon X\to X$ be a continuous map.
\begin{enumerate}
	\item If $G\subset X$ is a graph, then $f(G)$ does not contain any oscillatory quasi-arc.
	\item If $L$ and $K$ are two oscillatory quasi-arcs in $X$
  with $L\subset f(K)$, then
   the order of $L$ is not bigger than the order of $K$ and
          $\omega(L)\subset f(\omega(K))$.
\end{enumerate}
\end{lem}

Let $L$ be an oscillatory quasi-arc in a quasi-graph $X$.
For $t\in\R_+$, we will use  $L[t,\infty)$ to denote $\varphi ([t,\infty))$
with respect to a given parameterization $\varphi\colon \R_+\to L$. The following result is {\cite[Proposition 3.5]{MS17}}.

\begin{prop}\label{prop:k-order-OSC-arc}
Let $X$ be a quasi-graph and let $f\colon X\to X$ be a continuous map.
Suppose that $L$ and $K$ are $k$-order oscillatory quasi-arcs
in $X$ for some $k\geq 1$ and $\varphi\colon \R_+\to L$ and $\phi\colon \R_+\to K$
are parameterizations of $L$ and $K$ respectively.
If $L\subset f(K)$, then $f(\omega(K))=\omega(L)$ and
$f(K[s,\infty))=L[r,\infty)$ for some $r,s\in\R_+$.
\end{prop}

Before proceeding, firstly we extend \cite[Proposition 3.5]{MS17} as follows.
Two quasi-arcs $L,K$, with parameterizations
$\varphi,\phi \colon \R_+ \to X$ respectively, are called
\textit{eventually the same}
if there are $s,t\geq 0$ such that $\varphi([s,\infty))=\phi([t,\infty))$.

\begin{lem}\label{lem:exact-rank}
Let $X$ be a quasi-graph, $f\colon X\to X$ a continuous surjection and $L\subset X$ a $k$-order (with $k\ge 1$) oscillatory quasi-arc with a parameterization $\varphi\colon \R_+\to L$. Then:
\begin{enumerate}

\item \label{lem:exact-rank:firstitem}
There is an oscillatory quasi-arc $K$ in $X$
such that $L [a, \infty)\subset f(K)$ for some $a\in \R_+$.

\item \label{lem:exact-rank:seconditem}
If $K$ is an oscillatory quasi-arc in $X$ such that $L[a, \infty)\subset f (K)$ for some $a\in \R_+$, then the order of $K$ is exactly $k$,
$f(K)$ and $L$ are eventually the same and $\omega(L)= f(\omega(K))$.

\item \label{lem:exact-rank:thirditem} $f(L)$ contains a $k$-order oscillatory quasi-arc.
\end{enumerate}
\end{lem}
\begin{proof}
\textbf{\eqref{lem:exact-rank:firstitem}} Since $f$ is surjective, there is a sequence of points $x_n\in X$
such that $f(x_n)=\varphi(n)$.
Write $X=G\cup \bigcup_{j=1}^N L_j$ as in Theorem~\ref{thm:dec-qg}.
By Lemma~\ref{lem:img-qg}, $f(G)$ does not contain oscillatory quasi-arcs and so there is $r_1>0$ such that
$f(G)\cap \varphi([r_1,\infty))=\emptyset$, in particular, $x_n\notin G$ for all sufficiently large $n$.
Similarly, there is $s_1>0$ such that
$G\cap \varphi([s_1,\infty))=\emptyset$.
But then we must have $x_n\in L_i$ for some $i$ and infinitely many $n$.
By Proposition~\ref{prop:arcs-in-QG}, there exits $a\geq 0$
such that $L[a,\infty)\subset f(L_i)$. This finishes the proof of \eqref{lem:exact-rank:firstitem}.

\textbf{\eqref{lem:exact-rank:seconditem}}
Now assume that
$K$ is an oscillatory quasi-arc in $X$
  such that $A=L[a, \infty)\subset f (K)$ for some $a\in \R_+$. In the following we shall prove that the order of $K$ is exactly $k$, and then obtain the conclusion by applying Proposition~\ref{prop:k-order-OSC-arc}.

By Lemma~\ref{lem:img-qg} the order of $K$
cannot be smaller than the order of $A$. Note that the quasi-arcs $A$ and $L$ are eventually the same and then they have the same order, thus the order of $K$ is at least $k$. It suffices to prove that the order of $K$ is at most $k$.

Denote by $m$ the maximal order among oscillatory quasi-arcs of $X$.
Observe that each oscillatory quasi-arc in $X$ must be eventually the same to $L_\ell$ for some $\ell\in \{1, \dots, N\}$ by Proposition~\ref{prop:arcs-in-QG},
and then we have $m\ge k$. Now let us prove the conclusion by induction.

First we consider the case of $k=m$.
Clearly the order of $K$ is at most $k$ by the definition of $m$, as by Proposition~\ref{prop:arcs-in-QG} each oscillatory quasi-arc in $X$ has its order at most $m$.

Next let $n\geq 0$ be such that the result holds for
all quasi-arcs in $X$ with its order among $m-n,m-n+1,\ldots,m$
and assume that $k=m-n-1\geq 1$.

Fix any $j\in \{m-n,\ldots,m\}$ and any oscillatory quasi-arc $Q$ in $X$ with its order $j$. Now it suffices to prove $L\not\subset f(Q)$.
By \eqref{lem:exact-rank:firstitem} and our inductive assumptions, using Proposition \ref{prop:k-order-OSC-arc} we can construct in $X$ a finite sequence of oscillatory quasi-arcs $A_{-N -1}, A_{- N}, \dots, A_{- 1}, A_0$ of order $j$ such that $f(A_i)= A_{i+1}$ for all $i= -N- 1, -N, \dots, -1$ and $A_0=Q[r_0,\infty)$ for some $r_0> 0$. But by Proposition~\ref{prop:arcs-in-QG} each oscillatory quasi-arc in $X$ must be eventually the same to $L_\ell$ for some $\ell\in \{1, \dots, N\}$, hence we must have that $A_p$ and $A_q$ are eventually the same
for some indexes $p<q$. Applying Proposition~\ref{prop:k-order-OSC-arc} again we obtain that oscillatory quasi-arcs $A_{p+1}$ and $A_{q+1}$, as images of oscillatory quasi-arcs $A_p$ and $A_q$ respectively, are eventually the same, and then by induction, there is $t>0$ such that $A_{-t}$ and $A_0$ are eventually the same.
Note that $f(A_{-t})=A_{-t+1}$ and so
there exists $r_1\in\R_+$ such that $f(Q[r_1,\infty))$ and $A_{-t+1}$ are eventually the same. But then $f(Q[r_1,\infty))$
 is eventually
the same to $L_{\ell_1}$ for some $\ell_1\in\{1,\dotsc,N\}$, and $L_{\ell_1}$ has order $j$.
As $L$ has order $k$ and $\{L_i\}_{i=1}^N$ are pairwise disjoint,
$L$ is  eventually
the same to $L_{\ell_2}$ for some $\ell_2\in\{1,\dotsc,N\}\setminus\{\ell_1\}$.
Then by Proposition~\ref{prop:k-order-OSC-arc}, $L\not\subset f(Q[r,\infty))$ for any $r\in\R_+$.
This proves that $K$ must have order $k$.

\textbf{\eqref{lem:exact-rank:thirditem}}
By (1) and (2), for any $k$-order oscillatory quasi-arc $A$ in $X$,
there exists a $k$-order oscillatory quasi-arc $B$ in $X$
such that $f(B)$ and $A$ are eventually the same.
Note that there are only finitely many $k$-order oscillatory quasi-arcs
in $X$ up to identification of eventually the same oscillatory quasi-arcs.
Therefore, for each $k$-order oscillatory quasi-arc in $X$,
its image under $f$ must contain a $k$-order oscillatory quasi-arc.
\end{proof}

\subsection{Invariant measures and topological entropy for quasi-graph maps} \label{property:qg}
Now we are ready to
show  that every invariant measure of a map acting on a
quasi-graph is isomorphic to a finite convex combination of invariant measures on graphs.
Recall that for a dynamical system $(X,f)$ a point $x\in X$ is \textit{recurrent} if $\liminf_{n\to\infty}d(f^nx,x)=0$.

\begin{lem}\label{lem:division}
Let $X$ be a quasi-graph, $f\colon X\to X$ be a continuous map, and $k>1$ be the maximal order among oscillatory quasi-arcs of $X$.
Then there exists a quasi-graph map $(Y,f_1)$ and a graph map $(G,f_2)$ such that
\begin{enumerate}
\item[(1)]  the maximal order among oscillatory quasi-arcs of $Y$ is at most $k- 1$, and

\item[(2)] if the system $(X, f)$ has zero topological entropy then both $(Y,f_1)$ and $(G,f_2)$ have zero topological entropy.
\end{enumerate}
 Furthermore, if $\mu$ is an invariant measure of $(X,f)$, then
there exist:
\begin{enumerate}
	\item[(3)] invariant measures $\mu_1,\mu_2$ on $(X,f)$ and $\alpha\in[0,1]$ such that $\mu=\alpha \mu_1+(1-\alpha)\mu_2$, additionally, if $\alpha\in (0,1)$ then $\mu_1$ and $\mu_2$ are singular;
	\item[(4)] an invariant measure $\nu_1$ of $(Y,f_1)$ such that $(X, \mu_1, f)$
	and $(Y,\nu_1,f_1)$ are measure-theoretically isomorphic; and
	\item[(5)] an invariant measure $\nu_2$ of $(G,f_2)$ such that $(X, \mu_2, f)$
	and $(G,\nu_2,f_2)$ are measure-theoretically isomorphic once $\alpha\in [0, 1)$.
	\end{enumerate}
\end{lem}
\begin{proof}
Since we are dealing with invariant measures, by Proposition~\ref{prop:quasi-graph-onto}
we may assume that $f$ is surjective (replacing $X$ by $\bigcap_{n=0}^\infty f^n(X)$ if necessary).
Fix a presentation $X=G^*\cup \bigcup_{j=1}^n L_j$
provided by Theorem~\ref{thm:dec-qg}.
As each oscillatory quasi-arc in $X$ must be eventually the same to $L_l$ for some $l\in \{1, \dots, n\}$, by Lemma \ref{lem:exact-rank} we know that the closed set $Q:=\bigcup_{j=1}^n\omega(L_j)$
is $f$-invariant, i.e., $f(Q)= Q$.

\smallskip

\noindent\textbf{Claim.}
There is a quasi-graph $Y$ obtained by adding some arcs to $Q$
(and so $Y\supset Q$)
and a continuous surjection $f_1\colon Y\to Y$ such that
$f|_{Q}=f_1|_{Q}$.
\begin{proof}[Proof of Claim]
	By \cite[Proposition~2.31]{MS17}, every compact connected set
	in $X$ has at most $n+1$ arcwise connected components.
	Each $\omega(L_j)$ is connected, then $Q$ has finitely many
	arcwise connected components, which we enumerate as $X_0,X_1,\dotsc,X_{m-1}$.
	Note that the continuous image of every arcwise connected set is
	itself arcwise connected, therefore
	by the condition $f(Q)=Q$, there is a permutation $\tau$ on $\{0,1,\dotsc,m-1\}$
	such that $f(X_i)=X_{\tau(i)}$ for all $i=0,1,\dotsc,m-1$.
	Assume first that $f(X_i)=X_{i+1\pmod m}$.
	Fix a point $x\in X_0$. Then $f^i(x)\in X_{i}$ for all $i=1,2,\dotsc,m-1$
	and $f^m(x)\in X_0$.
	We construct $Y$ in the following way:
	add to $X_0$ an exterior point $z$ and then connect $z$ with $f^i(x)$
	by an arc $J_i$ for every $i=0,1,\dotsc,m-1$ in such a way
	that $J_i\cap J_j=\{z\}$ and $J_i\cap X_i=\{f^i(x)\}$.
	It is clear that $Y$ is arcwise connected and then a quasi-graph.
	We define a map $f_1\colon Y\to Y$ as follows.
	First we define $f_1$ as $f$ on $Q$.
	For $i=0,1,\dotsc,m-2$, $f_1$ maps
	the arc $J_i$ homeomorphically onto the arc $J_{i+1}$
	with $f_1(z)=z$ and $f_1(f^i(x))=f^{i+1}(x)$.
	If $f^{m}(x)=x$, $f_1$ maps
	the arc $J_{m-1}$ homeomorphically onto the arc $J_{0}$
	with $f_1(z)=z$ and $f_1(f^{m-1}(x))=x$.
	If $f^m(x)\neq x$, as $X_0$ is arcwise connected,
	we pick an arc $K\subset X_0$ with endpoints $x$ and $f^m(x)$.
	Then we define $f_1$ to map the arc $J_{m-1}$
	homeomorphically onto the arc $J_0\cup K$ in such a way
	that $f_1(z)=z$ and $f_1(f^{m-1}(x))=f^m(x)$.
	Then $f_1$ is well-defined and continuous.
	Since we can arrange the above homeomorphisms arbitrarily, we can require that $f_1$ is strictly monotone on each $J_i$,
	that is, if $x_*,f^m(x_*)\in Y\setminus Q$ and $x_*\neq z$ then $x_*\neq f^m(x_*)$ and the shortest arc $[z,x_*]$ connecting $z,x_*$ is contained in the arc $[z,f^m(x_*)]$.
	In particular, $z$ is the only recurrent point in $Y\setminus Q$.
	
	If $\tau$ has more than one cycle, we construct an appropriate quasi-graph for each cycle independently, and then combine them into one quasi-graph by identifying the fixed point $z$ in all of these independent quasi-graphs.
\end{proof}

Next we define a relation $\sim$ on $X$ by putting
$a \sim b$ provided that $a=b$ or
$a\in \omega(L_i)$ and $b\in \omega(L_j)$ for some indexes $i,j$.
The relation $\sim$ is clearly a closed equivalence relation.
Let $G=X/_\sim$ and $\pi \colon X\to G$ be the associated quotient mapping.
As the relation $\sim$ is invariant under $f\times f$,
we can naturally define a continuous map $f_2=f/_\sim$ on $G=X/_\sim$
such that $\pi\circ f=f_2\circ \pi$.
Since we collapsed $\omega$-limit sets of all oscillating quasi-arcs of $X$
to a single point, it is not hard to see that
$G$ is a graph.

Recall that in the proof of Claim, we attached to $Q$ a finite number of arcs,
obtaining the quasi-graph $Y$ that way.
So the maximal order of quasi-arcs in $Y$ is at most $k-1$.
Moreover, it is easy to see from the above construction that if $(X, f)$ has zero topological entropy then both $(Y,f_1)$ and $(G,f_2)$ have zero topological entropy. Furthermore, any invariant measure supported on $Q$ can be regarded as an $f_1$-invariant measure on $Y$ (by the condition $f|_{Q}=f_1|_{Q}$) and any invariant measure supported on $X\setminus Q$ can be regarded as an $f_2$-invariant measure on $G$ via the factor map $\pi$ (by the definition $\pi|_{X\setminus Q}$ is a one-to-one map).

Now assume that $\mu$ is an invariant measure of $(X,f)$, and denote $\alpha=\mu(Q)$.
If $\alpha= 1$, then set $\mu_1= \mu= \mu_2$.
If $0<\alpha<1$, then set $\mu_1= \frac{1}{\alpha} \mu|_{Q}$
and $\mu_2= \frac{1}{1-\alpha} \mu|_{X\setminus Q}$.
If $\alpha=0$, then let $\mu_1$ be any invariant
measure supported on $Q$ and set $\mu_2=\mu$.
Then the invariant measures $\nu_1$ and $\nu_2$ are defined naturally. It is easy to show that these invariant measures $\mu_1, \mu_2, \nu_1, \nu_2$ satisfy the required properties.
\end{proof}

If we inductively apply Lemma \ref{lem:division}
to the maximal order of oscillatory quasi-arcs of a quasi-graph map,
then we have the following main result of this section.

\begin{thm}\label{thm:quasi-graph-invarn-measure}
Let $X$ be a quasi-graph and let $f\colon X\to X$ be a continuous map.
Then there exist graph maps $(G_1,f_1), \ldots, (G_k, f_k)$ for some $k\in \mathbb{N}$ such that
\begin{enumerate}

\item  each invariant measure on $(X,f)$ is measure-theoretically
isomorphic to a finite convex combination of invariant measures
on these graph maps,

\item each ergodic invariant measure on $(X,f)$ is measure-theoretically
isomorphic to an ergodic invariant measure on $(G_i, f_i)$ for some $i= 1, \ldots, k$, and

\item
 if the system $(X, f)$ has zero topological entropy then all $(G_1,f_1), \ldots, (G_k, f_k)$ have zero topological entropy.
\end{enumerate}
\end{thm}

It is clear (cf. Theorem~\ref{thm:discrete}),
a finite convex combination of invariant measures which have discrete spectrum also has discrete spectrum.

\begin{proof}[Proof of Theorem~\ref{thm:quasi-graph-dis-spect}]
It is enough to combine Theorems~\ref{thm:graph-map-dis-spect} and \ref{thm:quasi-graph-invarn-measure}.
\end{proof}

%It was proved in \cite{HWZ16} that the M\"obius function is linear disjoint from a dynamical system if it admits only countably many ergodic invariant measures such that each of them has discrete spectrum, and then in \cite{HWY17} it was shown that the M\"obius function is linear disjoint from an invertible dynamical system if each of its invariant measure has discrete spectrum.
%Even though  dynamical systems considered in \cite{HWY17} are invertible, it is not hard see that the results also holds for non-invertible systems.
%
%\begin{proof}[Proof of Corollary~\ref{cor:MDC}]
%It is a direct consequence of Theorem~\ref{thm:quasi-graph-dis-spect} and the above discussion.
%\end{proof}

Let $X$ be a quasi-graph and let
$s\geq 2$.
An \textit{$s$-horseshoe} for $f\colon X\to X$ is a closed arc $I \subset X$ and
closed subarcs $J_1,\ldots,J_s\subset I$ with pairwise disjoint interiors,
such that $f(J_j) = I$ for all $j = 1,\ldots ,s$.
We shall denote this horseshoe by $(I;J_1,\dotsc,J_s)$.
An $s$-horseshoe is \textit{strong} if in addition the intervals $J_j$ are contained in the interior of $I$ and are pairwise disjoint.

Llibre and Misiurewicz proved the following result relating
positive topological entropy and the existence of horseshoes on graph maps (cf. \cite[Theorem B]{LM93}).

\begin{thm} \label{thm:LM93}
Let $G$ be a graph and let $f\colon G\to G$ be a continuous map.
Assume $\htop(f)>0$, where $\htop (f)$ denotes the topological entropy of $(G, f)$. Then there exist strictly increasing sequences $s_n,k_n$
of positive integers such that each $f^{k_n}$ has an $s_n$-horseshoe and
$\lim_{n\to\infty}\frac{1}{k_n}\log (s_n)=\htop(f)$.
\end{thm}

We show that this result also holds for quasi-graph maps.

\begin{thm}\label{thm:hors}
Let $X$ be a quasi-graph and let $f\colon X\to X$ be a continuous map.
Assume $\htop(f)$ $> 0$. Then there exist strictly increasing sequences $s_n,k_n$
of positive integers such that each $f^{k_n}$ has an $s_n$-horseshoe and
$\lim_{n\to\infty}\frac{1}{k_n}\log (s_n)=\htop(f)$.
\end{thm}

\begin{proof}
It suffices to construct strictly increasing sequences $s_n,k_n$
of positive integers such that each $f^{k_n}$ has an $s_n$-horseshoe and
$\liminf_{n\to\infty}\frac{1}{k_n}\log (s_n)\ge \htop(f)$.
We shall prove the conclusion by performing induction on the maximal order  of quasi-arcs in $X$.

If the maximal order of quasi-arcs in $X$ is zero, i.e., $X$ is a graph,
then the result is just Theorem~\ref{thm:LM93}.
Now assume that the result holds for all quasi-graphs
with quasi-arcs of order at most $k$.
Let $X$ be a quasi-graph with $k+ 1$ being the maximal order of quasi-arcs in $X$.
Fix a presentation $X=G^*\cup \bigcup_{j=1}^n L_j$ provided by Theorem~\ref{thm:dec-qg} and let $f\colon X\to X$ be
 a continuous map  satisfying $\htop(f)>0$.
Since the topological entropy of $f$ and the one of $f$ restricted to the closure of all recurrent points (so-called \textit{Birkhoff center}) are the same (e.g., see \cite{ALM}), and since
by Proposition~\ref{prop:quasi-graph-onto} the set $\bigcap_{n=0}^\infty f^n(X)$ is also a quasi-graph, without loss of generality we may assume that
$X=\bigcap_{n=0}^\infty f^n(X)$, in particular, $f$ is surjective.

Fix any $0< \eps< \htop (f)$.
By the variational principle %concerning entropy,
there exists  an ergodic invariant measure $\mu$ on $(X,f)$
such that $\htop(f)-\eps<h_\mu(f)$, where $h_\mu (f)$ denotes the measure-theoretic $\mu$-entropy of $f$. Set $Q= \bigcup_{j=1}^n\omega(L_j)$.
Then there are two possibilities: either
$\mu(Q)=0$ or $\mu(Q)=1$.

Let us consider firstly the case of $\mu(Q)=0$.
Following the proof of Lemma~\ref{lem:division}, we collapse the $\omega$-limit set $Q$
to a single point $p$, and get a factor map $\pi\colon (X,f)\to (G,g)$
by setting $\pi(Q)=\{p\}$.
Let $\nu= \pi (\mu)$, which in fact is measure-theoretically isomorphic to $\mu$.
Then $\htop(g)\geq h_\nu(g)=h_\mu(f)$.
As $G$ is a topological graph, by Theorem~\ref{thm:LM93},
there exist sufficiently large positive integers $s>4$ and $k$ such that
$g^k$ has an $s$-horseshoe $(I;J_1,\dotsc,J_s)$,
\[\frac{\log s}{k}> \htop(g)-\varepsilon>\htop (f)- 2 \varepsilon
\quad\text{ and }\quad \frac{\log 4}{k}<\varepsilon.\]
The point $p$ is contained in at most two intervals $J_i$.
Removing those intervals and shortening $I$ if necessary,
we get an $r$-horseshoe $(I'; J_1',\ldots, J_r')$ for $g^k$
such that $p\not\in I'$
and $r\geq \frac{s-2}{2}$.
Since the restricted factor map
$\pi\colon X\setminus Q\to G\setminus\{p\}$ is invertible,
we may view $(I'; J_1',\ldots, J_r')$ as an $r$-horseshoe for $f^k$.
Clearly $r>s/4$,
therefore
\[\frac{\log r}{k}\geq \frac{\log s}{k}-\frac{\log4}{k}
>\frac{\log s}{k}-\eps> \htop(g)-2\eps> \htop(f)-3\eps.\]
This concludes the first case.

Now, we consider the remaining case, when $\mu(Q)=1$.
Following the construction in the Claim in the proof of Lemma~\ref{lem:division},
we get a quasi-graph $Y$ and a continuous map $g\colon Y\to Y$ extending $f|_Q$, where $Y$ was obtained by adding some arcs
to $Q$.
It is clear from the construction that $g(Q)\subset Q$, so no point from $Q$
can map to  interiors of these new arcs.
By the definition of the map $g$, we know that $Y\setminus Q$ contains a fixed point $z$, which is the unique recurrent point of $g$ in $Y\setminus Q$.
By $f|_Q= g|_Q$ and the fact that topological entropy of a system is supported on the closure of its all recurrent points (cf. \cite{ALM}), we have $\htop(g)=\htop(Q,f)\geq h_\mu(f)$, where $\htop(Q,f)$ denotes the topological entropy of $f$ restricted to $Q$. As the maximal order of quasi-arcs of $Y$ is at most $k$, applying our inductive assumption we can choose integers $s$ and $m$ large enough such that $g^m$ has an $s$-horseshoe and $\frac{\log s}{m}> \htop (g)- \eps\ge h_\mu (f)- \eps> \htop (f)- 2 \eps$. We shall finish the proof by showing $I\subset Q$ once
$(I;J_1,\dotsc,J_s)$ is an $s$-horseshoe for $g^m$, because then $(I;J_1,\dotsc,J_s)$ is also an $s$-horseshoe for $f^m$ (here recall again that $f|_Q= g|_Q$). Clearly, we may assume that $s\geq 3$. We show firstly
that the arc $J_1$ is contained in $Q$.
Assume the contrary that $J_1\setminus Q\neq \emptyset$. Then $z\in J_1$, as otherwise by the strict monotonicity of the map $g$ we have $J_1\setminus g^{mk}(J_1)\neq \emptyset$.
% once $k$ is large enough, which is impossible by the definition of a horseshoe. Furthermore, by the construction of the quasi-graph $Y$ and the definition of the map $g$, it is not hard to show that $z$ will not be an ending point of the arc $J_1$, that is, $z$ belongs to the interior of $J_1$.
But then, since $s>2$, without loss of generality we may assume that $z\not\in J_2$ and so $z\not\in g^m(J_2)$ (by the definition of the map $g$), which is impossible since $J_1\subset g^m(J_2)$ by the definition of a horseshoe.
Indeed, the only possibility is that $J_1\subset Q$.
But then $g^m(J_1)\subset Q$ by the invariance of $Q$, and so $I\subset Q$, which finishes the proof.
\end{proof}

\begin{rem}
In \cite{LM93}, Llibre and Misiurewicz also proved
that the topological entropy, as a function of a continuous map of a given graph
is lower-semicontinuous (cf.\ \cite[Theorem~C]{LM93}).
In fact, as shown by the proof of Theorem~\ref{thm:hors} the same result also holds for quasi-graph maps, because if a quasi-graph map $f$ has a strong $s$-horseshoe with a closed arc $I$ then by the definition of
the supremum metric, all sufficiently small perturbations of $f$ also have an $s$-horseshoe (see \cite{LM93} for a detailed proof).
\end{rem}

\section{Dynamics on dendrite maps} \label{main}
In this section, we study dynamics on dendrite maps.
First we need the following result on the structure of
$\omega$-limit set in a special class of dendrite maps.

\begin{thm}[{\cite[Theorem 1.2]{Ghasen}}]\label{thm:decomp}
	Let $X$ be a dendrite such that
	$End(X)$ is a closed set with finitely many accumulation points and
	let $f\colon X\to X$ be a continuous map with zero topological entropy. If
	$L= \omega_f (x)$ is an uncountable $\omega$-limit set for some $x\in X$, then for every $k\geq 1$ there is an $f$-periodic subdendrite $D_k$ of $X$ and an integer $n_k\geq 2$  with the following
	properties:
	\begin{enumerate}
		\item $D_k$ has period $\alpha_k:=n_1 n_2 \dots n_k$ for every $k\geq 1$,
		\item $\bigcup_{k=0}^{n_j -1}f^{k \alpha_{j-1}}(D_{j}) \subset D_{j-1}$ for every $j\geq 2$,
		
		\item $L \subset \bigcup_{i=0}^{\alpha_k -1}f^{i}(D_k)$ for every $k\geq 1$,
		
		\item for every $k\geq 1$ and each $0\leq i \leq \alpha_{k}-1$, $f(L \cap f^{i}(D_k))=L\cap f^{i+1}(D_k)$ and in particular $L \cap f^{i}(D_k)\neq \emptyset$, and
		
		\item for every $k\geq 1$ and all $0\leq i\neq j<\alpha_k $, $f^{i}(D_k)\cap f^{j}(D_k)$ has an empty interior.
	\end{enumerate}
\end{thm}

Recall that a dynamical system $(X,f)$ is \emph{mean equicontinuous}
if for every $\eps>0$ there exists $\delta>0$ such that
for every $x,y\in X$,
\[
d(x,y)<\delta\ \text{implies}\ \limsup_{n\to\infty}\frac{1}{n}\sum_{i=0}^{n-1}d(f^i(x),f^i(y))<\eps.
\]
It is shown in \cite{LTY15} (see also \cite{DG16}) that if a system $(X,f)$
is mean equicontinuous then every orbit closure is uniquely ergodic
and its unique invariant measure has discrete spectrum.

It is clear that the following result is a strong version of Theorem~\ref{thm:dendrite-finite}.
\begin{thm}\label{thm:dendrite-finite2}
	Let $X$ be a dendrite such that
	$End(X)$ is a closed set having finitely many accumulation points and
	let $f\colon X\to X$ be a continuous map with zero topological entropy.
	Fix any point $x\in X$.
	\begin{enumerate}
		\item If $\omega_f(x)$ is at most countable,
		then every invariant measure on $(\overline{\orb_f(x)},f)$ has discrete spectrum.
		\item
		If $\omega_f(x)$ is uncountable,
		then $(\overline{\orb_f(x)},f)$ is mean equicontinuous.
		In particular, $(\overline{\orb_f(x)},f)$ is uniquely ergodic
		and its unique invariant measure has discrete spectrum.
	\end{enumerate}
\end{thm}
\begin{proof}
	(1) If $\omega_f(x)$ is at most countable,
	then $\overline{\orb_f(x)}$ is also at most countable.
	So every ergodic invariant measure on $(\overline{\orb_f(x)},f)$
	is an equidistributed measure on a periodic orbit.
	By Theorem~\ref{thm:discrete}, every invariant measure on $(\overline{\orb_f(x)},f)$ has discrete spectrum.
	
	(2) We use the notation from Theorem~\ref{thm:decomp}.
	By \cite[Lemma 6.1]{Ghasen}, there is $k>0$ and $0\leq i_0<\alpha_k$ such that $f^{i_0}(D_k)$ is a free arc.
	As $\omega_f(x)$ is uncountable and contained in $\bigcup_{i=0}^{\alpha_k -1}f^{i}(D_k)$,
	there exists some $0\leq i_1\leq \alpha_k-1$
	such that $\omega_f(x) \cap f^{i_1}(D_k)$ is uncountable.
	Since for each $0\leq i \leq \alpha_{k}-1$, $f(\omega_f(x) \cap f^{i}(D_k))=\omega_f(x)\cap f^{i+1}(D_k)$,
	we have for each $0\leq i \leq \alpha_{k}-1$, $\omega_f(x) \cap f^{i}(D_k)$ is uncountable.
	
	Let $I= f^{i_0}(D_k)$ and $g=f^{\alpha_k}$.
	Then $g\colon I\to I$ is an interval map.
	It is clear that $g$ also has zero topological entropy.
	As $\omega_f(x)\cap I$ is uncountable, $\omega_f(x)$ has a nonempty intersection with interior of the interval $I$, then there exists $q\in\mathbb{N}$
	such that $f^q(x)\in I$.
	Let $y=f^q(x)$. Then $\omega_g(y)$ is either a finite set
	or an uncountable set without periodic points (see e.g. \cite[Proporsition 5.23]{R17}).
	Note that $\omega_f(x)=\bigcup_{i=0}^{\alpha_k-1}f^i(\omega_g(y))$.
	So the only possibility is that $\omega_g(y)$ is an uncountable set without periodic points, and $\omega_f(x)$ does not contain periodic points.
	By \cite[Lemma 2.2]{Ghasen}, for $0\leq i\neq j<\alpha_k $, the subset $f^{i}(D_k)\cap f^{j}(D_k)$ contains at most one point.
	As $f^{i}(D_k)\cap f^{j}(D_k)$ is $f^{\alpha_k}$-invariant,
	$f^{i}(D_k)\cap f^{j}(D_k)$ consists of a fixed point of $f^{\alpha_k}$ if it is not empty.
	Then
	$f^{i}(D_k)\cap f^{j}(D_k) \cap \overline{\orb_f(x)} =\emptyset$.
  By \cite[Lemma 3.3]{LOYZ17}, $(\overline{\orb_g(y)},g)$ is mean equicontinuous.
  Then for every $i=0,1,\dotsc,\alpha_k-1$, $(\overline{\orb_g(f^i(y))},g)$
  is also mean equicontinuous.
  As $\overline{\orb_g(f^i(y))}\subset f^i(f^{i_0}(D_k))$
  and $\overline{\orb_g(f^i(y))}\cap \overline{\orb_g(f^j(y))}
  =\emptyset$ for $i\neq j$,
  one has $(\bigcup_{i=0}^{\alpha_k-1} \overline{\orb_g(f^i(y))},g)$
  is mean equicontinuous.
  Note that $\bigcup_{i=0}^{\alpha_k-1} \overline{\orb_g(f^i(y))}
  =\overline{\orb_f(y)}\cup\{x,f(x),\dotsc,f^{q-1}(x)\}$,
  $(\overline{\orb_f(x)},g)$ is mean equicontinuous and then so is
  $(\overline{\orb_f(x)},f)$.
\end{proof}

As a consequence of Theorem~\ref{thm:dendrite-finite2}, we obtain another proof of \cite[Theorem 4.10]{Marzougi}.
% by applying directly Theorem~\ref{discrete} and .

\begin{cor}
	Let $X$ be a dendrite such that
	$End(X)$ is a closed set with finitely many accumulation points and
	let $f\colon X\to X$ be a continuous map with zero topological entropy.
	Then the M\"obius function is linearly
	disjoint from the system $(X,f)$.
\end{cor}

\begin{rem}
	In Theorem~\ref{thm:dendrite-finite2},
	it would be interesting to know that whether every orbit closure
	is uniquely ergodic, in other words,
	if $\omega_f(x)$ is at most countable, does there exist only one
	periodic orbit in $\omega_f(x)$?
\end{rem}

Finally, we provide a proof of Theorem~\ref{thm:Gehman-dendrite} and Corollary~\ref{cor:factors}, respectively.

\begin{proof}[Proof of Theorem~\ref{thm:Gehman-dendrite}]
Let $(V,\sigma)$ be a one-sided subshift over a finite alphabet $\mathcal{A}$.
Observe, that without loss of generality we may assume that $V$ admits no isolated points. Namely, if it was not the case, let us consider its product with any infinite, uniquely ergodic proximal subshift. There are numerous techniques to obtain such subshifts, e.g. see Examples 8.3--8.7 in \cite{KKO}. Clearly such subshift has %entropy zero and
a unique fixed point, hence the product
has the same simplex of invariant measures as $(V,\sigma)$ (there is natural isomorphism between invariant measures of $(V,\sigma)$ and the newly constructed product system), $(V,\sigma)$ is its subsystem (as the fiber over fixed point) and the product system can be regarded as a one-sided subshift over a finite alphabet $\mathcal{A}'\subset \mathcal{A}\times \mathcal{A}$.

By a standard technique, on the Gehman dendrite $D_*$ we can construct a surjective map $f_*$, such that there is a point $c\in D_*$ (fixed by $f_*$)
	and an $f_*$-invariant set $\Sigma \subset D_*\setminus \{c\}$ such that $(\Sigma, f_*)$ is conjugated to the one-sided full shift on $2$ symbols
	and for every $x\in D_*\setminus \Sigma$ we have $(f_*)^m(x)=c$ for some $m\in \mathbb{N}$ (e.g., see \cite[Example~6]{Kocan}).
	Fix a positive integer $n$ such that $2^n$ is larger than the cardinality of $\mathcal{A}$ (resp. $\mathcal{A}'$). Now over the Gehman dendrite $D_*$ we consider the map $f= (f_*)^n$: in $D_*\setminus \{c\}$ there exists an $f$-invariant set $\Omega$ such that $(\Omega, f)$ is conjugated to the one-sided full shift on $2^n$ symbols (and hence we may view $(V, \sigma)$ as a subsystem of $(\Omega, f)$), furthermore, $f (c)= c$ and for every $x\in D_*\setminus \Omega$ we have $f^m (x)= c$ for some $m\in \mathbb{N}$.
	
	As $D_*$ is a dendrite, for any two different points $x_1, x_2\in D_*$ there exists a unique arc connecting them (denote it by $[x_1, x_2]$). By above discussions we may view $V\subset \Omega$, and then we consider an arcwise connected $f$-invariant
	set $D=\bigcup_{y\in V} [y, c]$, where $f$ acts naturally over $D$. As $V$ does not contain isolated points, $D$ is again the Gehman dendrite (see \cite[Theorem~4.1]{Charatonik}).
	Then $(D, f)$ contains $(V, \sigma)$ as a subsystem, and is a surjective dynamical system over a Gehman dendrite, where $f$ maps $[y, c]$ exactly to $[f (y), c]$ for every $y\in V$. Furthermore, by the above construction we know that for every $x\in D\setminus V$ we have $f^m (x)= c$ for some $m\in \mathbb{N}$, and then besides of ergodic invariant measures for $(V, h)$, the system $(D, f)$ admits only one more ergodic invariant measure whose support is the fixed point $c$. %, hence topological entropy of $(D, f)$ is the same as that of $(V, \sigma)$.
\end{proof}

\begin{proof}[Proof of Corollary~\ref{cor:factors}]
Fix any dynamical system $(X,f)$ with zero topological entropy.
In \cite{Boyle02} Boyle \textit{et al} showed that
$(X,f)$
is a factor of a two-sided subshift with zero topological entropy $(U,g)$ (see also \cite[Theorem~6.9.9]{Downar}).
If $(V, h)$ denotes one-sided subshift generated by $(U, g)$, that is,
$$V= \{(w_0, w_1, w_2, \ldots): (\ldots, w_{- 1}, w_0, w_1, w_2, \ldots)\in U\},$$
then it is not hard to check that if the M\"obius function is linearly
disjoint from $(V,h)$ then it is disjoint from $(X,f)$. Simply, we may view $(U,g)$ as a shift homeomorphism on the natural extension (inverse limit)  generated by $(V,h)$.
Since disjointness with M\"obius function is a heredetary property of subsystems, the result follows by Theorem~\ref{thm:Gehman-dendrite}.
\end{proof}

%\begin{ques}
%Let $X$ be a dendrite such that
%$End(X)$ is  closed set having finitely many accumulation points and
%let $f\colon X\to X$ be a dendrite map with zero topological entropy. Does every invariant measure have discrete spectrum?
%\end{ques}

\section*{Acknowledgements}

This work was partially done during a series of visit of J. Li and P. Oprocha to the School of Mathematical Sciences, Fudan University. They both gratefully acknowledge the hospitality of Fudan University. We thank T. Downarowicz and R. Zhang for many fruitful discussions on properties of maps with zero topological entropy, and thank El~H.~El~Abdalaoui for bringing \cite{Glasner} into our attention. The authors are grateful to M.~Lemanczyk for numerous remarks on invariant measures with discrete spectrum.

J. Li was  supported by NSFC
Grant no.\ 11771264 and NSF of Guangdong Province (2018B030306024). P. Oprocha was supported by National Science Centre, Poland (NCN), grant no. 2015/17/B/ST1/01259
and by the Faculty of Applied Mathematics AGH UST statutory tasks within subsidy of Ministry of Science and Higher Education.
G. Zhang was supported by NSFC
Grants no. 11671094, 11722103 and 11731003.

\end{document}